\documentclass[11pt, oneside]{amsart}   	

\usepackage{geometry}                		
\usepackage{graphicx}				
\usepackage{caption}					
\usepackage{amsmath,amssymb,amsthm}
\usepackage{color}
\newcommand{\red}[1]{{\color{black} #1}}

\usepackage[normalem]{ulem}
\usepackage{cancel}

\usepackage{tikz}
\usepackage{pgfplots}
\usepackage{subfigure}
\usepackage{subcaption}
\usepackage{bm}

\numberwithin{equation}{section}

\pgfplotsset{compat=1.13}
\usetikzlibrary{shapes,positioning,intersections,quotes}
\usepackage{hyperref}

\newtheorem{theorem}{Theorem}[section]

\newtheorem{lemma}[theorem]{Lemma}

\newtheorem{definition}[theorem]{Definition}
\newtheorem{assumption}[theorem]{Assumption}
\newtheorem{remark}[theorem]{Remark}

\newcommand{\B}{\mathcal{B}}

\newcommand{\R}{\mathbb{R}}
\newcommand{\Rcp}{\mathcal{R}_{\text{cp}}}

\newcommand{\N}{\mathbb{N}}

\newcommand{\Sph}{\mathbb{S}}

\newcommand{\A}{\mathcal{A}}

\newcommand{\cS}{\mathcal{S}}
\newcommand{\cT}{\mathcal{T}}

\newcommand{\cN}{\mathcal{N}}

\newcommand{\Omgdel}{\Omega_{\delta}}
\newcommand{\rmLip}{\mathrm{Lip}}
\newcommand{\rmext}{\mathrm{ext}}
\newcommand{\rmGN}{\mathrm{GN}}
\newcommand{\rmBer}{\mathrm{Ber}}
\newcommand{\rmFB}{\mathrm{FB}}
\newcommand{\rmSamp}{\mathrm{Samp}}
\newcommand{\rmStab}{\mathrm{Stab}}

\newcommand{\M}{\mathbb{M}}

\renewcommand{\d}{\mathrm{d}}

\newcommand{\cC}{\mathcal{C}}

\DeclareMathOperator{\dist}{dist}

\newcommand{\Hawaii}{Hawai\kern.05em`\kern.05em\relax i }

\newcommand{\dco}{d_c}

\title[Inverse inequalities]{Inverse inequalities for kernel-based approximation on bounded domains and Riemannian manifolds}
\author{Zhengjie Sun}
\address{School of Mathematics and Statistics,  Nanjing University of Science and Technology, Nanjing, China}
\email{zhengjiesun@njust.edu.cn}

\author{Leevan Ling}
\address{Department of Mathematics, Hong Kong Baptist University, Kowloon Tong, Hong Kong}
\email{lling@hkbu.edu.hk}

\date{\today}

\begin{document}

	\begin{abstract}
		This paper establishes inverse inequalities for kernel-based approximation spaces defined on bounded Lipschitz domains in $\mathbb{R}^d$ and compact Riemannian manifolds. While inverse inequalities are well-studied for polynomial spaces, their extension to kernel-based trial spaces poses significant challenges.
		For bounded Lipschitz domains, we extend prior Bernstein inequalities, which only apply to a limited range of Sobolev orders, \red{to the full range of lower and upper orders}, and derive Nikolskii inequalities that bound $L_\infty$ norms by $L_2$ norms.
		For compact Riemannian manifolds, we focus on restricted kernels, which are defined as the restriction of positive definite kernels from the ambient Euclidean space to the manifold, and prove their counterparts.
	\end{abstract}
	\keywords{Bernstein inequality, Nikolskii inequality, restricted kernel, manifolds, sampling inequalities}
	\subjclass{
		41A17, 
		65D05,  
		65D12. 
}
\maketitle
%
\section{Introduction}
\label{S_Introduction}
Inverse inequalities (or inverse estimates) are fundamental tools in numerical analysis with broad applications in finite element methods \cite{dahmen-2004MCoM-inverse,georgoulis-2008MCoM-inverse,graham-2005IMAJNA-finite} and approximation theory \cite{hangelbroek-2018MCoM-inverse,Narcowich+WardETAL-SoboErroEstiBern:06,schaback-2002MCoM-inverse,ward-2012JAT-lp}. These inequalities establish rigorous relationships between functional norms within finite-dimensional approximation spaces, providing essential foundations for analyzing numerical stability and convergence. Specifically, they bound stronger norms (e.g., Sobolev norms) by weaker norms (e.g., $L_2$ norms), with explicit dependence on discretization parameters such as mesh size or polynomial degree.

Classical inverse inequality theory for polynomial spaces is well-established \cite{borwein2012polynomials,dai-2006JFA-multivariate,ditzian-2016ConstrApprox-nikolskii,migliorati-2015JAT-multivariate,milovanovic1994topics,Nikolskii1974}, tracing back to foundational work by Markov and Bernstein \cite{borwein2012polynomials,cheney-2009course}. Despite their theoretical utility, polynomial approximations face inherent limitations including oscillatory behavior in high degrees and geometric constraints. The growing importance of manifold-based problems in machine learning, geometric analysis, and data science has consequently driven interest in more flexible approaches. Kernel-based methods have emerged as powerful alternatives due to their geometric adaptability
\cite{chen2020extrinsic,fasshauer2015kernel,fuselier2013high,jayasumana-2015TPAMI-kernel,lin-2024SINUM-kernel,lin-2025MCoM-distributed,narcowich2017novel,Sun_2025IMAJNA_spherical,Wendland,wendland-Adv-2020solving,wenzel-2024SISC-data}.

Extending inverse estimates to kernel-based approximation spaces presents significant challenges. A clear understanding of the existing landscape of inverse inequalities is crucial for identifying these challenges, motivating new work, and contextualizing advances. To this end, Section~\ref{Sec:InvIneqReview} provides a comprehensive review of the current state of the art, including Bernstein inequalities, Nikolskii inequalities, and inverse theorems. This review highlights the progress made in radial basis function (RBF) approximation, and establishes a foundation for our contributions while outlining essential directions for future research.
While substantial progress has been achieved in deriving inverse estimates for RBF approximation in Euclidean spaces, generalizing these results to bounded Lipschitz domains and complex manifolds remains problematic. For general manifolds, intrinsically defined positive definite kernels enable the construction of Lagrange and local Lagrange functions \cite{hangelbroek2018direct,lehoucq2016meshless,sun-2024SISC-high}, which can be instrumental in establishing inverse inequalities. However, these kernels typically lack closed-form representations because they arise as fundamental solutions to elliptic operators of the form
$\mathcal{L}=\sum_{j=0}^m(\nabla_{\M}^j)^*\nabla_{\M}^j$, where $\nabla_{\M}$ denotes the covariant derivative and $(\nabla^j_{\M})^*$ its adjoint operator.  Deriving such fundamental solutions is highly nontrivial, even for relatively simple manifolds. Additionally, the absence of a closed-form atlas for $\M$ often precludes solving the coordinate representation of the fundamental solution equation. Moreover, in applications such as learning theory and computational fluid dynamics, the underlying manifolds are frequently unknown, poorly characterized, or exhibit high geometric and topological complexity, making it difficult to construct intrinsically geometric kernels.

As a practical alternative, researchers have developed the restriction approach for constructing positive definite kernels on manifolds.  Given $\phi_m:\R^{d} \times \R^{d} \to \R$, the kernel restricted to the $d_\M$-dimensional manifold $\M\subset\R^d$ is defined as
$\psi_\tau(\cdot,\cdot):=\phi_m(\cdot,\cdot)|_{\M\times\M}$ for $\tau=m-(d-d_\M)/2$.
This approach has been rigorously investigated in the literature. Narcowich et al. \cite{narcowich2007approximation} provided a connection between the Fourier transforms of radial kernels and the Fourier-Legendre coefficients of their spherical restrictions. Fuselier and Wright \cite{fuselier2012scattered} derived error estimates for scattered data interpolation on embedded submanifolds. These results demonstrate that restricted kernels preserve positive definiteness and key approximation properties while offering practical utility in applications. Nevertheless, a critical gap persists: \emph{inverse estimates for restricted kernels remain unestablished}.

The first objective is to establish inverse inequalities within finite-dimensional approximation spaces constructed using positive definite
Sobolev space reproducing kernels in bounded Lipschitz domains and
their restrictions on Riemannian manifolds. For bounded Lipschitz domains, we employ an inverse inequality established in \cite{Cheung+LingETAL-leaskerncollmeth:18}, which connects $H^m(\Omega)$ and $H^{\alpha}(\Omega)$ for $\alpha \in (d/2, m]$ with $\alpha, m \in \mathbb{N}$, and extend this to the case $\alpha, m \in \mathbb{R}$. By combining this with a Gagliardo–Nirenberg-type interpolation inequality, we derive \red{the full} Bernstein inequality that bounds the $H^{\alpha}(\Omega)$-norm \red{for any real $\alpha \in [0, m]$} of trial functions in $V_{X,\phi_m,\Omega}$ in terms of their \red{$H^t(\Omega)$-norm for $t\in [0,\alpha]$.} Furthermore, using the sampling inequality and a recently proposed stability result from \cite{wenzel2026sharp}, \red{we establish stability results and derive the Nikolskii inequality.}

The remainder of this work is structured as follows. Section~\ref{sec:backgroud} introduces the necessary notation, including manifold geometry, kernel methods, Sobolev spaces, and point set distributions.
A survey of prior work in Section~\ref{Sec:InvIneqReview} motivates the challenges addressed in this paper. We provide a detailed review of the development of inverse estimates in kernel approximation, covering Bernstein inequalities and Nikolskii inequalities.
Section~\ref{sec:NikolskiiIneq_BoundedDomains} focuses on bounded Lipschitz domains, developing Bernstein-type inequalities through an interpolation inequality and a stability result that connects continuous and discrete norms.
Section~\ref{sec:Nikolskii_Manifold} extends these analytical tools to restricted kernels on manifolds. We introduce a diffeomorphic map that connects the embedded manifold to its ambient space. This framework allows us to transfer kernel properties from the ambient Euclidean space to the manifold, enabling the derivation of analogous results for restricted kernels, and yields the Bernstein and Nikolskii inequalities in the manifold setting.
Finally, we conclude the paper in Section~\ref{sec:Conclusion}.

\section{Notation and preliminaries}
\label{sec:backgroud}
In this section, we provide basic background on the manifold and kernels used in this article. Throughout this paper, let $\M\subset \R^{d}$ be a smooth embedded manifold with the induced Riemannian metric. We denote the dimension of $\M$ by $d_{\M}$. We consider a finite node set $X=\{x_1,x_2,\ldots,x_N\}$ in a metric space $\cS$. The \emph{mesh norm} (or \emph{fill distance}) and the \emph{separation radius}  of the points are defined as
$$h_{X,\cS}:=\sup_{x\in\cS}\min_{x_j\in X}\text{dist}_{\cS}(x,x_j),\quad q_{X,\cS}:=\frac{1}{2}\min_{x_j,x_k\in X, j\neq k}\text{dist}_{\cS}(x_j,x_k),$$
where $\text{dist}_{\cS}(x,y)$ is the distance between points $x$ and $y$ intrinsic to $\cS$.
The \emph{mesh ratio} is then defined as $\rho_{X,\cS}:=h_{X,\cS}/q_{X,\cS}$.

\red{
	An open, bounded, and connected set $\Omega \subset \mathbb{R}^d$
	is said to be a Lipschitz domain if its boundary
	$\partial \Omega$ is locally the graph of a Lipschitz function, with $\Omega$ lying locally on one side. To quantify our results, we use the following definition from \cite{kenig_2011CPAM_layer,kim_2024CVPDE_neumann} (see also~\cite[p.~189]{stein1970singular}).
	\begin{definition}
		Let $\Omega \subset \mathbb{R}^d$ be an open and connected set.
		Then $\Omega$ is a Lipschitz domain with character $\rmLip_\Omega:=(L,N)$
		for $L \ge 0$ and $N \in \mathbb{N}$ with $a_0<N<\infty$, if there exist $r=r_0>0$
		and $z_j \in \partial \Omega$ for $j \in \{1,\dots,N\}$, such that
		$\partial \Omega \subset \bigcup_{j=1}^{N} B_r(z_j)$.
		Moreover, for each $j \in \{1,\dots,N\}$, there exists a Lipschitz function $\varphi_j : \mathbb{R}^{d-1} \to \mathbb{R}$ with
		$\varphi_j(0)=0$
		and
		$\|D\varphi_j\|_\infty \le L$
		such that, after rotation and translation, $z_j=0$ and
		$$
		B_{a_0(L+1)r}(z_j)\cap \Omega
		=
		B_{a_0(L+1)r}(z_j)
		\cap
		\left\{
		(x',x_n)\in\mathbb{R}^d
		:
		x'\in\mathbb{R}^{d-1},\;
		x_n>\varphi_j(x')
		\right\}.
		$$
	\end{definition}
}


We analyze functions in Sobolev spaces \cite{Narcowich+WardETAL-SoboErroEstiBern:06}. 
\red{For a bounded domain $\Omega\subset\R^d$ and for $k\in\N_0$, we can also define the Sobolev space $W_2^{k}(\Omega)$ in the weak-derivative sense by}
$W_2^{k}(\Omega):=\{f\in L_2(\Omega):\|f\|_{W_2^{k}(\Omega)}<\infty\}$ via the Sobolev norm
$$\|f\|_{W_2^{k}(\Omega)}:=\Big(\sum_{|\nu|\leq k}\|D^{\nu}f\|_{L_2(\Omega)}^2\Big)^{1/2},
\quad \mbox{and} \quad 
\|f\|_{W_{\infty}^{k}(\Omega)}:=\max_{|\nu|\leq k}\|D^{\nu}f\|_{L_{\infty}(\Omega)}.$$
For fractional order $m=k+t$ with $k\in\N_0$ and $0<t<1$, we define
$$\|f\|_{W_2^{k+t}(\Omega)}=\Big(\|f\|_{W_2^k(\Omega)}^2+\sum_{|\nu|=k}\int_{\Omega}\int_{\Omega}\frac{|D^{\nu}f(x)-D^{\nu}f(y)|^2}{|x-y|^{d+2t}}\d x\d y\Big)^{1/2}.$$
\red{In particular, we write $H^m(\Omega):=W_2^m(\Omega)$ for $m\ge 0$.} \red{Throughout, we use the (fractional) Sobolev space $H^{m}(\R^d)$ equipped with the Fourier-based norm
	$$
	\|f\|_{H^{m}(\R^d)}
	=\Bigl(\int_{\R^d}(1+|\xi|^2)^{m}\,|\widehat f(\xi)|^2\,\d\xi\Bigr)^{1/2}.
	$$}

On manifolds, Sobolev spaces can be defined in several equivalent ways. Here, we define them using an atlas. Let $\tilde{\A}=\{(\tilde{U}_j,\tilde{\varphi}_j)\}_{j=1}^N$ be an atlas of slice charts for $\M$, and let $\A=\{U_j,\varphi_j\}_{j=1}^N$ be the associated intrinsic atlas. Let $\{\chi_j\}$ be a partition of unity subordinate to $\{\tilde{U}_j\}$. For a function $f$ defined on $\M$, we define the projections $\pi_j(f):\R^{d_{\M}}\rightarrow\R$ as
\begin{equation}\label{eq:DefProj}
	\pi_j(f)(y)=
	\begin{cases}
		\chi_j (f)(\varphi^{-1}_j(y)),~~&\text{if}~y\in B'_{r_j}(0),\\
		0,~~&\text{otherwise},
	\end{cases}
\end{equation}
where $B'_{r_j}(0)$ is a ball in $\R^{d_{\M}}$ corresponding to the slice chart $\tilde{U}_j$.
With this construction, the Sobolev space $W_p^m(\M)$ for $1 \leq p < \infty$ and $m\in\N_0$ can be defined via the norm
$$\|f\|_{W_p^{m}(\M)}:= \Big(\sum_{j=1}^N\|\pi_j(f)\|_{W_p^m(\R^{d_\M})}^p\Big)^{1/p}.$$
While the norm depends on the choice of atlas $\tilde{\A}$ and the partition of unity ${\chi_j}$, different choices lead to equivalent norms and define the same Sobolev space.
For any function defined on a finite set $X$ of $N$ points, we define the discrete norm as
\begin{equation*}
	\|f\|_{\ell_\varrho(X)}=\left\{
	\begin{aligned}
		&\Big(\sum_{j=1}^{N}|f(x_j)|^\varrho\Big)^{1/\varrho},~~&&\text{if}~1\leq \varrho<\infty,\\
		&\max_{x_j\in X}|f(x_j)|,~~&&\text{if}~\varrho=\infty.
	\end{aligned}\right.
\end{equation*}
We note that our definition of the discrete norm differs from certain conventions in the literature that include a normalization factor of $1/N$ on the right-hand side \cite{wenzel2026sharp}.

Outside of theorem statements, for non–negative quantities \(A\) and \(B\), we sometimes use the notation
$A \;\lesssim\; B
\Longleftrightarrow
A \le c\,B,$
where the multiplicative factor $c>0$ is independent of the
discretization parameters. Additionally, the constant $C=C_{\,*\,}$ (for example, $C_{d,\phi_m}$) is recorded to indicate its precise dependence on fixed
geometric or analytic data. Moreover, $A\sim B$ indicates that there exist two generic constants $0<C_1<C_2<\infty$ such that $C_1 B\leq A\leq C_2 B$. \red{We also use $\cC_*$ (for example, $\cC_\rmext$) to denote special generic constants, which will be introduced and used later in the paper.}

\subsection{Sobolev space reproducing kernel and restricted kernel}
A common class of positive definite kernels is the RBFs, which take the form $\phi(x, y) = \phi(\|x - y\|)$ depending only on the distance between their arguments. The decay of the Fourier transform of a kernel $\phi$ determines the smoothness of the functions in its associated reproducing kernel Hilbert space, often referred to as the native space. For example, if the Fourier transform $\widehat{\phi}_m(\xi)$ satisfies the decay condition
\begin{equation}\label{eq:FourierDecay}
	\red{c_1}(1+\|\xi\|_2^2)^{-m}\leq \widehat{\phi}_m(\xi)\leq \red{c_2} (1+\|\xi\|_2^2)^{-m}  \quad\mbox{for some }m>d/2,
\end{equation}
\red{for two constants $0<c_1\leq c_2<\infty$}, then the native space $\cN_{\phi_m}$ of the kernel $\phi_m$ is equivalent to the Sobolev space $H^m(\R^{d})$, \red{i.e., 
	\begin{equation}\label{eq:norm_equiv_sobolev}
		\|f\|_{\cN_{\phi_m}(\R^d)}\leq c_1^{-1/2}\|f\|_{H^m(\R^d)} \quad \text{and}\quad \|f\|_{H^m(\R^d)}\leq c_2^{1/2}\|f\|_{\cN_{\phi_m}(\R^d)}.
	\end{equation}
	Since the constants $m, c_1, c_2$ in \eqref{eq:FourierDecay} are determined by $\phi_m$, we use $C_{\phi_m}$ to denote constants arising from the norm 
	equivalence \eqref{eq:norm_equiv_sobolev} between the native and Sobolev spaces throughout the paper.
}

Given the kernel $\phi_m$ and a finite discrete set $X\subset\Omega\subset\R^{d}$, we can construct the finite-dimensional approximation space as
\begin{equation}\label{eq:finite_dim_approx_space}
	V_{X,\phi_m,\Omega}:=\text{span}\{\phi_m(\cdot,x_j):x_j\in X\}.
\end{equation}
To extend the framework of positive definite kernels to manifolds $\psi_\tau(\cdot,\cdot):=\phi_m(\cdot,\cdot)|_{\M\times\M}$,  we consider the restriction of kernels defined on the ambient space $\R^{d}$ onto a compact manifold $\M \subset \R^{d}$ of dimension $d_{\M}$, and can define the approximation space $V_{X,\psi_\tau,\M}$ similarly as
\begin{equation}\label{eq:approx_space_restricted}
	V_{X,\psi_\tau,\M}:=\text{span}\{\psi_\tau(\cdot,x_j):x_j\in X\subset \M\}
	=\text{span}\{\phi_m(\cdot,x_j):x_j\in X\subset \M\}.
\end{equation}
The smoothness of the functions in $V_{X,\psi_\tau,\M}$ depends solely on the smoothness of the kernel $\psi_\tau$. This property allows us to construct approximation spaces with arbitrary smoothness by selecting kernels with the desired level of smoothness. Consequently, this flexibility enables the development of approximation spaces that achieve arbitrarily high approximation orders. The restricted kernel $\psi_\tau$ inherits positive definiteness from $\phi_m$ and induces a native space $\cN_{\psi_\tau}$ on $\M$ \cite{adams_2003_sobolev,besov1979integral,fuselier2012scattered,Wendland}.

To formulate a connection between Sobolev norms and discrete norms, an inverse inequality is essential for the finite-dimensional approximation space $V_{X,\phi_m}$, which enables the bounding of the Sobolev norm of functions in $V_{X,\phi_m}$ by a discrete $\ell_2$ norm \cite{schaback-1995AdvComput-error,schaback-2002MCoM-inverse,wenzel2026sharp}. Here, we provide the proof below for completeness.

\begin{lemma}\label{lem:InvIneq}
	Let $\Omega\subset\R^{d}$ be a bounded Lipschitz domain and $\phi_m$ be the reproducing kernel for $H^m(\Omega)$ satisfying the decay condition \eqref{eq:FourierDecay}.  Then for any set of pairwise distinct points $X\subset\Omega$, there exists a constant \red{$\cC_{\rmStab}=C_{d,\phi_m}>0$} such that
	$$\|u\|_{H^m(\Omega)}\leq \red{\cC_{\rmStab}}\, q_{X,\Omega}^{d/2-m}\|u\|_{\ell_2(X)}
	$$
	holds for all trial functions $u\in V_{X,\phi_m,\Omega}$.
\end{lemma}

\red{
	\begin{proof}
		By \cite[Corollary 10.48]{Wendland}, each trial function $u\in V_{X,\phi_m,\Omega}$ has a natural extension $E_{\cN} u\in \cN_{\phi_m}(\R^d)=H^m(\R^d)$, simply by evaluation of $u$ at points throughout $\R^d$. Hence
		\begin{equation}\label{eq:natural_extension_norm}
			\|u\|_{H^m(\Omega)}\le \|E_{\cN} u\|_{H^m(\R^d)}\le \cC_{\phi_m}\|E_{\cN} u\|_{\cN_{\phi_m}(\R^d)}=\cC_{\phi_m}\|u\|_{\cN_{\phi_m}(\Omega)}.
		\end{equation}
		Let  $A_X=\big(\phi_m(\|x_k-x_j\|)\big)$, $u|_X=\big(u(x_1),\ldots,u(x_N)\big)^T$, and let $\lambda_j$ denote the eigenvalue of $A_X$. Then, we have
		$$\|u\|_{\cN_{\phi_m}(\Omega)}^2=u|_X^T A_X^{-1}u|_X\leq \Big(\min_{j} \lambda_j\Big)^{-1}\|u\|_{\ell_2(X)}^2.$$
		Moreover, by \cite[Corollary 12.8]{Wendland}, there exists a constant $C_{d,m}>0$ depending only on $d$ and $m$ such that
		$$\min_{j} \lambda_j\geq C_{d,m}q_{X,\Omega}^{2m-d}.$$
		Combining these estimates completes the proof.
	\end{proof}
}

\smallskip

\section{Survey of inverse inequalities in kernel-based approximation}
\label{Sec:InvIneqReview}
Direct error estimates for scattered data approximation in kernel-based spaces, particularly for radial basis functions, are well-established and widely applied in numerical analysis and scientific computing \cite{Buhmann,fasshauer2007meshfree,Wendland}. However, the theory of inverse estimates for such approximation spaces remains underdeveloped, especially for bounded Lipschitz domains. This section provides a systematic review of the current landscape, highlighting foundational results, key challenges, and geometric limitations that motivate our work and directions for future research.

\subsection{Bernstein inequality} The derivation of inverse estimates often relies on Bernstein-type inequalities in finite-dimensional function spaces, which establish bounds for strong norms (e.g., Sobolev norms) in terms of weaker Lebesgue norms.

Bernstein inequalities are relatively well-developed in boundary-free settings. In Euclidean spaces, for kernel-based approximation spaces $V_{X,\phi_m,\R^{d}}$ associated with kernels $\phi_m$ satisfying the decay condition \eqref{eq:FourierDecay}, Narcowich et al. \cite{Narcowich+WardETAL-SoboErroEstiBern:06} established the fundamental Bernstein inequality for $s\in[0,m]$ that:
\begin{equation}\label{eq:FirstBernIneq}
	\|u\|_{H^{s}(\R^{d})}\leq C  q_{X,\R^{d}}^{-s}\|u\|_{L_2(\R^{d})},
	\quad\mbox{for all }u\in V_{X,\phi_m,\R^{d}}.
\end{equation}
Subsequently, Ward \cite{ward-2012JAT-lp} extended this inequality to general $L_p$ norms. For spherical basis functions (SBFs) defined on the unit sphere $\Sph^d$, analogous Bernstein results from $H^s(\Sph^d)$ to $L_2(\Sph^d)$ have been developed in \cite{narcowich-2007FoCM-direct}.
Mhaskar et al. \cite{mhaskar_2010MCoM_LpBernstein} further generalized these results, deriving $L_p$ Bernstein estimates in Bessel-potential Sobolev spaces for SBFs.

Extending Bernstein-type inequalities to general bounded domains remains highly challenging due to boundary effects that disrupt the smooth interpolation properties of kernels. Rieger \cite{rieger-2008phd-sampling} addressed this issue by introducing the concept of scaled domains to mitigate boundary effects, which is defined as
$\Omega^{-\vartheta}:=\{(1-\vartheta)x:~x\in\Omega\}$.
Given a discrete set $X\subset\Omega^{-2q_X}\subset\Omega$ and kernel $\phi_m$ satisfying Fourier decay property \eqref{eq:FourierDecay}, Rieger proved a Bernstein inequality from $H^m(\R^d)$ to $L_2(\Omega)$.
Although this inverse inequality applies to bounded domains, it is limited by the requirement that the center points must lie within the scaled domain. Moreover, Griebel et al. \red{\cite{griebel-2015SINUM-multiscale,griebel-2018FoCM-regularized}} studied reproducing kernels associated with Sobolev spaces, where the kernels form a tight frame in a Hilbert space. They established an inverse inequality for this setting, showing that the rate is optimal for such kernels. However, this does not surpass the best-known rates for standard Sobolev spaces, and such kernels are rarely used in practice.

Significant progress has been made in this direction by Hangelbroek et al. \cite{hangelbroek2011kernel,hangelbroek2010kernel,hangelbroek2026kernel}, who demonstrated the existence of a family of intrinsic kernels $\kappa_{m,\M}(\cdot,\cdot)$ that are well-suited for interpolation on compact, connected, and smooth Riemannian manifolds $\M\subset\R^{d}$. Given a finite set of quasi-uniformly distributed data sites $X\subset\M$, the Lagrange basis function $\chi_{\xi}$ centered at $\xi\in X$ satisfies the Kronecker delta property $\chi_{\xi}(\zeta)=\delta_{\xi,\zeta},~ \chi_{\xi}\in\text{span}_{\xi\in X}\{\kappa_{m,\M}(\cdot,\xi)\}.$ Hangelbroek et al. \cite{hangelbroek2010kernel} further demonstrated that, for a specific class of kernels, these Lagrange functions are uniformly bounded and decay exponentially away from their centers. To apply the above results for a bounded Lipschitz domain $\Omega\subset\R^{d}$, further modifications are necessary to account for boundary effects. Specifically, the data set $X\subset\Omega$ must be extended to an enlarged set $X\subset\widetilde{X}\subset\R^{d}$ (see more details in \cite[Section 2.3]{hangelbroek-2018MCoM-inverse}).  Using Lagrange functions
$\chi_{\xi,\widetilde{X}}$ generated by kernels $\kappa_{m,\R^{d}}$ over $\widetilde{X}$, and letting the approximation space be defined with the Lagrange functions centered at the original point set $X$ as
$\text{span}\{ \chi_{\xi,\widetilde{X}}  \}_{\xi \in X}$,
Hangelbroek et al. \cite{hangelbroek-2018MCoM-inverse} established the $L_p$ Bernstein inequality for this function space. In a subsequent study, Cheung et al. \cite{Cheung+LingETAL-leaskerncollmeth:18} formulated a Bernstein-type inequality that connects the $H^m(\Omega)$ to $H^{\alpha}(\Omega)$ for integers $\alpha\in\N$ with $d/2< \alpha\leq m$.

\red{A complete theory of Bernstein‑type inequalities from $H^{\alpha}(\Omega)$ to $H^t(\Omega)$ ($0\leq t\leq \alpha\leq m$) for real $\alpha\in\R$ on bounded domains $\Omega$ remains an open problem. In addition, analogous inverse
	estimates on Riemannian manifolds remain largely unexplored, and their
	development is an open problem
	\cite{wendland-Adv-2020solving}. These inequalities are important in
	applications to meshless methods
	\cite{hangelbroek2018direct,wenzel2026sharp}, in particular for meshless
	discretizations of PDEs on surfaces
	\cite{Cheung+LingETAL-leaskerncollmeth:18,wendland-Adv-2020solving}. }


\subsection{Nikolskii inequality}
The Nikolskii inequality is an essential tool in approximation theory that establishes critical relationships between different (quasi-)norms of a function. Nikolskii-type inequalities have been extensively studied in classical settings, such as algebraic polynomials, trigonometric polynomials, and spherical harmonics 
\cite{dai-2013book-approximation,ditzian-2016ConstrApprox-nikolskii,ganzburg-2017ConApp-sharp,kamzolov1984approximation,mhaskar_2010MCoM_LpBernstein,nessel-1978JAMS-nikolskii,Nikolskii1974,zygmund2002trigonometric}.
In the following, we shall review the recent development of the Nikolskii inequality in kernel-based spaces.

The Nikolskii inequality on the sphere has garnered significant interest due to the absence of boundary effects and the availability of spherical polynomials and specialized basis functions. In the context of thin-plate splines or positive definite SBFs, K\"{u}nemund et al. \cite{kunemund2019high} derived a Nikolskii inequality.  Specifically, for a quasi-uniform point set $X$ on $\Sph^d$ and the associated finite-dimensional approximation space $V_{X,\kappa_m,\Sph^d}:=\text{span}_{\xi\in X}\{\chi_{\xi}\}$ built with the Lagrange functions, the following Nikolskii-type inequality holds:
\begin{equation}\label{eq:Nikolskii_On_sphere}
	\|u\|_{L_q(\Sph^d)}\leq C h_{X,\Sph^d}^{-d\big(\frac{1}{p}-\frac{1}{q}\big)_{+}}\|u\|_{L_p(\Sph^d)},~~u\in V_{X,\kappa_m,\Sph^d}.
\end{equation}

For more general settings, let $\M$ be a compact $d$-dimensional Riemannian manifold, and $\kappa_{m,\M}$ be the kernel considered in \cite{hangelbroek2010kernel} with $m>d/2$, Hangelbroek et al. \cite{hangelbroek2011kernel} proved a similar Nikolskii inequality on manifolds.
As mentioned in the introduction, the kernels used here are intrinsically defined on manifolds and have limited practical applications.

For a bounded Lipschitz domain $\Omega\subset\R^d$, using the extended point set $\widetilde{X}$ to construct the Lagrange functions and considering the approximation space $V_{\widetilde{X},\kappa_m,\Omega}$,
Hangelbroek et al. \cite{hangelbroek-2018MCoM-inverse} derived the Nikolskii inequality for bounded domains. The emergence of localized kernel methods in the work of Fuselier, Narcowich, and others \cite{fuselier2013localized,hangelbroek2018direct,kunemund2019high,narcowich2017novel} has further advanced the understanding of these inequalities. These methods enable the construction of local Lagrange functions using only neighborhood points. Approximation spaces generated by such local Lagrange functions exhibit Bernstein-type and Nikolskii-type inequalities, as shown in \cite{hangelbroek-2018MCoM-inverse,hangelbroek2018direct,hangelbroek-2024generalized,kunemund2019high}.

However, extending these Nikolskii-type inequalities to more general manifolds and even bounded domains remains an open challenge, as noted by Wendland and K\"{u}nemund \cite{wendland-Adv-2020solving}. Such extensions would have significant implications, particularly for meshless Galerkin approximations, as highlighted in \cite[Theorem 2.3]{kunemund2019high} and \cite[Theorem 2.8]{wendland-Adv-2020solving}.

\smallskip

\section{Inverse inequalities in bounded Lipschitz domains}
\label{sec:NikolskiiIneq_BoundedDomains}
In this section, we present new results on inverse inequalities for kernel-based approximation spaces defined on bounded Lipschitz domains. Our approach is built upon three key components: a Bernstein-type inequality for kernel spaces, a sampling inequality for Sobolev seminorms in bounded domains, and a recently developed discrete stability result that establishes a connection between discrete and continuous function norms.

\red{
	To lay the groundwork, we first state a result about the Stein-type extension operator. The proof can be found in \cite{stein1970singular} for integers $m\in\mathbb{N}$ and further generalized to the fractional case (see, \cite[p.~374]{brenner2008mathematical}, \cite[p.~181]{Narcowich+WardETAL-SoboErroEstiBern:06}).
	\begin{lemma}\label{lem:extension}
		Assume that $\Omega$ is a bounded Lipschitz domain. Then there exists a linear extension operator $E_{\Omega}$ such that,  
		for any real $s\ge 0$, we have $E_{\Omega} f|_\Omega =f$ for every $f\in H^s(\Omega)$ and 
		$E_{\Omega}:H^s(\Omega)\rightarrow H^{s}(\R^d)$ is bounded, that is, 
		$$\|E_{\Omega} f\|_{H^s(\R^d)}\leq \cC_{\rmext,s}\|f\|_{H^s(\Omega)},$$
		for some $\cC_{\rmext,s}=C_{s,d,\rmLip_\Omega}>0$, where $\rmLip_\Omega$ is the Lipschitz character of $\Omega$.
	\end{lemma}
}

\red{
	Next, we recall the Gagliardo--Nirenberg inequality, a classical result in the interpolation theory of Sobolev spaces; see, e.g., \cite[Theorem~1]{brezis2018gagliardo}. To make the dependence of the constants explicit, we include a brief proof.
	\begin{lemma}\label{lem:interp_lipschitz}
		Let $\Omega\subset \R^d$ be a bounded Lipschitz domain. Let $0\leq s_1<s_2$  and assume $s=\theta s_1+(1-\theta)s_2$, with $\theta\in(0,1)$.
		Then there exists a constant $\cC_{\rmGN}=C_{s_1,s_2,s,d,\rmLip_\Omega}>0$ such that for all $u\in H^{s_2}(\Omega)$, the following inequality holds
		$$
		\|u\|_{H^{s}(\Omega)}\le \cC_{\rmGN}\,
		\|u\|_{H^{s_1}(\Omega)}^{\theta}\,
		\|u\|_{H^{s_2}(\Omega)}^{1-\theta}.
		$$
	\end{lemma}
}

\red{
	\begin{proof}
		For any $v\in H^{s_2}(\R^d)\subset H^{s_1}(\R^d)$, we have
		$$
		(1+|\xi|^2)^s|\widehat v(\xi)|^2
		=\Bigl((1+|\xi|^2)^{s_1}|\widehat v(\xi)|^2\Bigr)^{\theta}
		\Bigl((1+|\xi|^2)^{s_2}|\widehat v(\xi)|^2\Bigr)^{1-\theta}.
		$$
		Integrating over $\R^d$ and applying H\"older's inequality yields 
		$$\begin{aligned}
			\|v\|_{H^{s}(\R^d)}^2
			&=\int_{\R^d}(1+|\xi|^2)^s|\widehat v(\xi)|^2\,\d\xi\\
			&\le
			\Bigl(\int_{\R^d}(1+|\xi|^2)^{s_1}|\widehat v(\xi)|^2\,\d\xi\Bigr)^{\theta}
			\Bigl(\int_{\R^d}(1+|\xi|^2)^{s_2}|\widehat v(\xi)|^2\,\d\xi\Bigr)^{1-\theta}\\
			&=\|v\|_{H^{s_1}(\R^d)}^{2\theta}\,
			\|v\|_{H^{s_2}(\R^d)}^{2(1-\theta)}.
		\end{aligned}$$
		
		
		By Lemma~\ref{lem:extension}, there exists a linear
		extension operator $E_{\Omega}$ which is bounded from $H^s(\Omega)$ into
		$H^s(\mathbb R^d)$ for every $s$ between $s_1$ and $s_2$. For any $u\in H^{s_2}(\Omega)$, we obtain
		$$
		\begin{aligned}
			\|u\|_{H^{s}(\Omega)}
			&\le \|E_{\Omega} u\|_{H^{s}(\R^d)}
			\le
			\|E_{\Omega} u\|_{H^{s_1}(\R^d)}^{\theta}\,
			\|E_{\Omega} u\|_{H^{s_2}(\R^d)}^{1-\theta}\\
			&\le
			\Bigl(C_{\rmext,s_1}\,\|u\|_{H^{s_1}(\Omega)}\Bigr)^{\theta}\,
			\Bigl(C_{\rmext,s_2}\,\|u\|_{H^{s_2}(\Omega)}\Bigr)^{1-\theta}\\
			&= \cC_{\rmGN}\,
			\|u\|_{H^{s_1}(\Omega)}^{\theta}\,
			\|u\|_{H^{s_2}(\Omega)}^{1-\theta},
		\end{aligned}
		$$
		where $\cC_{\rmGN} = C_{\rmext,s_1}^\theta C_{\rmext,s_2}^{1-\theta}=C_{s_1,s_2,s,d,\rmLip_\Omega}$. This completes the proof since $\theta$ is determined by $s,s1,s2$.
	\end{proof}
}

We collect several basic properties of band-limited functions, the proofs of which can be found in \cite[Section 3]{Narcowich+WardETAL-SoboErroEstiBern:06}.
\begin{lemma}\label{lem:Bandlimitedfunc}
	\red{Let $\sigma>0$ and $\B_{\sigma}:=\{f\in L_2(\R^d):\text{supp}~\widehat{f}\subseteq B_\sigma(0)\}$.
		For any real $\beta\in[0,m]$,
		every $f_\sigma\in\B_{\sigma}$ satisfies the Bernstein inequality
		\begin{equation}\label{eq:BandlimitedProp_3}
			\|f_{\sigma}\|_{H^m(\R^d)}
			\leq 2^{({m-\beta})/{2}}    \max\{1, \sigma^{m-\beta}\}    \|f_{\sigma}\|_{H^{\beta}(\R^d)}.
		\end{equation}
	}
	Moreover, let $m,\alpha\in \R^+$ with $d/2<\alpha\leq m$. Assume that $\Omega$ is a bounded Lipschitz domain,
	for any $u\in H^{m}(\Omega)$, there exists
	$ f_{\sigma,u,\alpha,\Omega}\in \B_{\sigma}:=\{f\in L_2(\R^d):\text{supp}~\widehat{f}\subseteq B_\sigma(0)\}$
	with $\sigma=\kappa_{\alpha,d}\, q_{X,\Omega}^{-1}$ for some $\kappa_{\alpha,d}>0$ that interpolates $u$ on $X$,
	\begin{equation}\label{eq:BandlimitedProp_1}
		u|_X = f_{\sigma, u, \alpha,\Omega}|_X,
		\quad
	\end{equation}
	and satisfies
	\begin{equation}\label{eq:BandlimitedProp_2}
		\|f_{\sigma, u, \alpha,\Omega}\|_{H^{\alpha}(\R^d)}
		\leq \red{C_{\alpha,d,\rmLip_\Omega}}\|u\|_{H^{\alpha}(\Omega)}.
	\end{equation}
\end{lemma}
\red{\begin{proof}
		Property \eqref{eq:BandlimitedProp_2} is stated in \cite[Corollary 3.5]{Narcowich+WardETAL-SoboErroEstiBern:06}, whose constant (given by $C_{\alpha,d,\rmLip_\Omega}=10\kappa^{-t}_{\alpha,d}\,\cC_{\rmext}$ with $t=
		0$) originates from \cite[Lemma~3.4]{Narcowich+WardETAL-SoboErroEstiBern:06}. Although \eqref{eq:BandlimitedProp_3} is not stated explicitly there, it follows directly from the definition of the
		Sobolev norm,
		\begin{align*}
			\|f_\sigma\|_{H^m(\R^d)}^2
			&
			=\int_{|\xi|\le\sigma}(1+|\xi|^2)^m\,|\widehat f_\sigma(\xi)|^2\,\d\xi \\
			&=\int_{|\xi|\le\sigma}(1+|\xi|^2)^{m-\beta}(1+|\xi|^2)^{\beta}\,|\widehat f_\sigma(\xi)|^2\,\d\xi \\
			&\le (1+\sigma^2)^{m-\beta}\,\|f_\sigma\|_{H^\beta(\R^d)}^2.
		\end{align*}
		Since $1+\sigma^2\le 2\max\{1,\sigma^2\}$, we obtain
		$(1+\sigma^2)^{(m-\beta)/2}
		\le 2^{(m-\beta)/2}\max\{1,\sigma^{m-\beta}\}$,
		which proves \eqref{eq:BandlimitedProp_3}.
	\end{proof}
}

We impose the following assumptions on the domain, the discrete set, and the kernel, which will be frequently used in the subsequent analysis.
\begin{assumption}\label{Assump1}
	Assume that $\Omega\subset\R^{d}$ is a bounded Lipschitz domain. Assume further that $X\subset\Omega$ is a finite set with the fill distance $h_{X,\Omega}$,  separation distance $q_{X,\Omega}$, and mesh ratio $\rho_{X,\Omega}$.
	Finally, let $V_{X,\phi_m,\Omega}$ denote the finite-dimensional approximation space spanned by the translates of the kernel $\phi_m$, centered at the nodes in $X$; the kernel $\phi_m$ satisfies the Fourier–decay condition~\eqref{eq:FourierDecay}.
\end{assumption}

With Lemma~\ref{lem:Bandlimitedfunc}, the authors in \cite{Cheung+LingETAL-leaskerncollmeth:18} established a Bernstein-type inverse inequality that relates $\|u\|_{H^m(\Omega)}$ to $\|u\|_{H^{\alpha}(\Omega)}$, where $\alpha$ is an integer satisfying $d/2 < \alpha \leq m$. 
\red{We recall this result and provide a concise proof in the following lemma, since the argument will be 
	adapted to establish the full Bernstein inequality, first for a bounded domain (in this section) and 
	then for the manifold (in the next).}
\begin{lemma}[\textbf{A Bernstein inequality}]\label{lem:H2invineq}
	Suppose Assumption~\ref{Assump1} holds.
	Then for any real $\alpha\in (d/2,m]$, there exists a constant \red{$\cC_{\rmBer}=C_{\alpha,m,d,\phi_m,\rmLip_\Omega}>0$} such that
	$$\|u\|_{H^m(\Omega)}\leq \red{\cC_{\rmBer}}\, q_{X,\Omega}^{-m+\alpha}\|u\|_{H^{\alpha}(\Omega)}$$
	holds for all trial functions $u\in V_{X,\phi_m,\Omega}$.
\end{lemma}
\begin{proof}
	By Lemma~\ref{lem:Bandlimitedfunc}, for any $u_m\in V_{X,\phi_m,\Omega}\subset H^m(\Omega)\subseteq H^{\alpha}(\Omega)$ with $\alpha\in (d/2,m]$, there exists a band-limited function
	$f_\sigma=f_{\sigma,u_m,\alpha,\Omega}\in\B_\sigma$ with
	$\sigma=\kappa_{\alpha,d}\,q_{X,\Omega}^{-1}$ such that
	\eqref{eq:BandlimitedProp_3}-\eqref{eq:BandlimitedProp_2} hold and
	$u_m=I_{X,\phi_m}f_\sigma$ on $\Omega$.
	Therefore, \red{by \eqref{eq:natural_extension_norm} and the orthogonality of the interpolant $I_{X, \phi_m} f_\sigma$ in the native space \cite[Corollary 10.24]{Wendland}, we have
		\begin{align*}
			\|u_m\|_{H^m(\Omega)}&\le \cC_{\phi_m}\|u_m\|_{\cN_{\phi_m}(\Omega)}
			\le \cC_{\phi_m}\Big(\|u_m-f_\sigma\|_{\cN_{\phi_m}(\Omega)}+\|f_\sigma\|_{\cN_{\phi_m}(\Omega)}\Big) \\
			&= \cC_{\phi_m}\Big(\|I_{X,\phi_m}f_\sigma-f_\sigma\|_{\cN_{\phi_m}(\Omega)}+\|f_\sigma\|_{\cN_{\phi_m}(\Omega)}\Big)\\
			& \leq \cC_{\phi_m}\cdot 2\|f_\sigma\|_{\cN_{\phi_m}(\Omega)}\leq 2\cC_{\phi_m}\|f_\sigma\|_{\cN_{\phi_m}(\R^d)}.
		\end{align*}
	}
	To further bound the right-hand side, we use the properties \eqref{eq:BandlimitedProp_3} and \eqref{eq:BandlimitedProp_2} to obtain
	\begin{align*}
		\|f_\sigma\|_{\red{\cN_{\phi_m}(\R^d)}}&\leq \red{\cC_{\phi_m}}\|f_\sigma\|_{H^m(\R^d)}
		\red{\le \cC_{\phi_m}2^{(m-\alpha)/2}\kappa_{\alpha,d}^{m-\alpha}} q_{X,\Omega}^{-(m-\alpha)}\|f_\sigma\|_{H^{\alpha}(\R^d)}\\
		& \red{\le \cC_{\phi_m}2^{(m-\alpha)/2}\kappa_{\alpha,d}^{m-\alpha}} \,
		q_{X,\Omega}^{-m+\alpha}\red{\cC_{\rmext,\alpha}}\|u_m\|_{H^{\alpha}(\Omega)}.
	\end{align*}
	This completes the proof \red{with $\cC_{\rmBer}=2\cC_{\phi_m}^2 2^{(m-\alpha)/2}\kappa_{\alpha,d}^{m-\alpha}\cC_{\rmext,\alpha}$}.
\end{proof}

A limitation of the above result is that the right-hand side cannot be expressed in terms of $\|u\|_{H^t(\Omega)}$ for $t\in[0,d/2]$, particularly in the case of $\|u\|_{L_2(\Omega)}$, which is a common feature in traditional Bernstein inequalities. To address this, we employ interpolation inequalities in Lemma~\ref{lem:interp_lipschitz} to derive \red{the full} Bernstein inverse inequalities for kernel spaces associated with Sobolev spaces of order $\alpha\in (d/2,m]$ without any extra smoothness assumption.


\begin{theorem}[\textbf{\red{Full} Bernstein inequality}] \label{thm:InvEstimate}
	Suppose Assumption~\ref{Assump1} holds. Then, for any real $\alpha\in[0,m]$ and real $t\in [0,\alpha]$, there exists \red{$\cC_{\rmFB}=C_{\alpha,t,d,\phi_m,\rmLip_\Omega}\ge 1$} such that
	\begin{equation}\label{eq:BernsteinIneq_Domain}
		\|u\|_{H^{\alpha}(\Omega)}
		\leq \red{\cC_{\rmFB}} q_{X,\Omega}^{-\alpha+t}\|u\|_{H^t(\Omega)}
	\end{equation}
	holds for all trial functions $u\in V_{X,\phi_m,\Omega}$.
\end{theorem}

\begin{proof}
	Let $s \in (d/2,m)$ be any real number. \red{For any trial function $u\in V_{X,\phi_m,\Omega}$, we use Lemma~\ref{lem:interp_lipschitz} (with $s_1=t$ and $s_2=m$) and Lemma~\ref{lem:H2invineq} to obtain}
	\begin{align*}
		\|u\|_{H^{s}(\Omega)}\le & ~\red{\cC_{\rmGN}}\|u\|_{H^t(\Omega)}^{1-\theta}\|u\|_{H^{m}(\Omega)}^{\theta}\\
		\le &~\red{\cC_{\rmGN}} \|u\|_{H^t(\Omega)}^{1-\theta}\Big(\red{\cC_{\rmBer}} q_{X,\Omega}^{-m+s}\|u\|_{H^{s}(\Omega)}\Big)^{\theta},
	\end{align*}
	where the two generic constants are defined in the corresponding lemmas, respectively.
	Simplifying, we get
	$\|u\|_{H^{s}(\Omega)}^{1-\theta} \le \red{\cC_{\rmGN} \cC_{\rmBer}^{\theta}} q_{X,\Omega}^{(-m+s)\theta}\|u\|_{H^t(\Omega)}^{1-\theta}$.
	Then, taking the $(1-\theta)$-th root leads to
	\begin{equation}\label{eq:InvEst_CL2}
		\begin{aligned}
			\|u\|_{H^{s}(\Omega)}
			\leq&~\red{(\cC_{\rmGN} \cC_{\rmBer}^\theta)^{1/(1-\theta)}} q_{X,\Omega}^{-s+t}\|u\|_{H^t(\Omega)}, ~~0\leq t<s,
		\end{aligned}
	\end{equation}
	where we have used the definition of $s$ to get $\frac{(-m+s)\theta}{1-\theta}=(t-m)\theta=-s+t$.
	
	\red{Fix $s=(m+d/2)/2$, so that any $s$-dependent constant becomes $C_s=C_{m,d}$. Then, Lemma 4.5 and \eqref{eq:InvEst_CL2} together 
		yield
		$$\|u\|_{H^m(\Omega)}\leq C q_{X,\Omega}^{-m+t}\|u\|_{H^t(\Omega)},~~0\leq t\leq m,$$ with the trivial case $t=m$ also covered by choosing $C=\max\{\cC_{\rmBer},\cC_{\rmBer} (\cC_{\rmGN} \cC_{\rmBer}^\theta)^{1/(1-\theta)}\}\geq 1$ and this $\cC_{\rmGN}=C_{(m+d/2)/2,t,d,\rmLip_{\Omega}}$ now depends on $m$, $t$, $d$.} 
	
	\red{Finally, the desired estimate \eqref{eq:BernsteinIneq_Domain} for general $\alpha\in [0,m]$ follows by
		interpolating between $H^t(\Omega)$ and $H^m(\Omega)$. Let 
		$\alpha:=\vartheta m+(1-\vartheta)t$, where $\vartheta\in(0,1)$. Then, there exists   $\tilde{\cC}_{\rmGN}=C_{\alpha,t,m,d,\rmLip_\Omega}$ such that
		\begin{align*}
			\|u\|_{H^{\alpha}(\Omega)}\le &~\tilde{\cC}_{\rmGN}\|u\|_{H^t(\Omega)}^{1-\vartheta}\|u\|_{H^{m}(\Omega)}^{\vartheta}\\
			\leq&~ \tilde{\cC}_{\rmGN}\|u\|_{H^t(\Omega)}^{1-\vartheta}\Big(C q_{X,\Omega}^{-m+t}\|u\|_{H^{t}(\Omega)}\Big)^{\vartheta}.
		\end{align*}
		The result then follows by rearranging the terms on the right-hand side.
	}
\end{proof}

	Next, by applying the sampling inequality from \red{Arcang{\'e}li} et al. \cite{arcangeli-2012NumerMath-extension}, we can establish the Nikolskii inequality. For a detailed discussion of how the constants depend on the
	relevant parameters, we refer the reader to
	\cite{rieger-2010Adv-sampling}.
	\red{
		\begin{lemma}[\textbf{Sampling inequalities from \cite[Theorem~3.1 and Theorem~3.2]{arcangeli-2012NumerMath-extension}}]\label{lem:sampling_ineq}
			Let $\Omega\subset\R^{d}$ be a bounded Lipschitz domain, which satisfies
			an interior cone condition with radius $r>0$ and angle $\omega\in (0,\pi/2)$, and let $X\subset\Omega$ be a finite set with the fill distance $h_{X,\Omega}<h_0$ being sufficiently small. Suppose that $d/2<m<\infty$,
			then for any real $s$ with $0\leq s\leq \lfloor m\rfloor$, there exists a constant $\cC_{\rmSamp,2}=C_{m,s,\omega,r,d}>0$
			such that
			\begin{equation}\label{eq:SamplingIneq1}
				|u|_{H^{s}(\Omega)}\leq \cC_{\rmSamp,2}\big(h_{X,\Omega}^{m-s}\|u\|_{H^{m}(\Omega)}+h_{X,\Omega}^{d/2-s}\|u\|_{\ell_2(X)}\big)
			\end{equation}
			holds for all functions $u\in H^{m}(\Omega)$.
			Furthermore, for any integer $k\in\N$ such that $0\leq k\leq \lceil m-d/2\rceil-1$, there exists a constant $\cC_{\rmSamp,\infty}=C_{m,k,\omega,r,d}>0$ such that
			\begin{equation}\label{eq:SamplingIneq2}
				|u|_{W_{\infty}^k(\Omega)}\leq \cC_{\rmSamp,\infty}\big(h_{X,\Omega}^{m-k-d/2}\|u\|_{H^{m}(\Omega)}+h_{X,\Omega}^{-k}\|u\|_{\ell_2(X)}\big)
			\end{equation}
			holds for all  functions $u\in H^{m}(\Omega)$.
		\end{lemma}
	}

	As a byproduct of the above sampling inequality, we obtain the following stability results. While these results do not play a direct role in this paper, we present them here for potential future applications.
	\begin{theorem}[\textbf{Stability on domains}]\label{thm:Cont_to_Disc_Domain}
		Suppose Assumption~\ref{Assump1} holds.
		Then for any real $s$ with $0\leq s\leq \lfloor m\rfloor$, \red{ there exists a constant $C=C_{m,s,\omega,r,d,\phi_m}>0$}
		such that
		\begin{equation}\label{eq:InvIneqHkNorm0}
			|u|_{H^{s}(\Omega)}\leq C\rho_{X,\Omega}^{m-d/2} h_{X,\Omega}^{d/2-s}\|u\|_{\ell_2(X)}
		\end{equation}
		holds for all trial functions $u\in V_{X,\phi_m,\Omega}$.  
		\red{In addition, there exists another constant $C=C_{m,\omega,r,d,\phi_m}>0$ such that}
		\begin{equation}\label{eq:InvIneqLinf0}
			\|u\|_{L_{\infty}(\Omega)}\leq C\rho_{X,\Omega}^{m-d/2}\|u\|_{\ell_2(X)},
		\end{equation}
		holds for all trial functions $u\in V_{X,\phi_m,\Omega}$.
	\end{theorem}
	
	\begin{proof}
		On the right-hand side of \eqref{eq:SamplingIneq1}, we can use the inverse inequality from Lemma~\ref{lem:InvIneq} to bound the term $\|u\|_{H^{m}(\Omega)}$, which gives
		\begin{equation}\label{eq:DiscInvIneq_Hm_ell2}
			\|u\|_{H^{m}(\Omega)}
			\leq \red{\cC_{\rmStab}}
			q_{X,\Omega}^{d/2-m}\|u\|_{\ell_2(X)}.
		\end{equation}
		Substituting \eqref{eq:DiscInvIneq_Hm_ell2} into \eqref{eq:SamplingIneq1}, we obtain
		\begin{equation*}
			\begin{aligned}
				|u|_{H^s(\Omega)}\le &~\red{\cC_{\rmSamp,2}} \Big(\red{\cC_{\rmStab}} h_{X,\Omega}^{m-s}q_{X,\Omega}^{d/2-m}\|u\|_{\ell_2(X)}+h_{X,\Omega}^{d/2-s}\|u\|_{\ell_2(X)}\Big)\\
				\le &~\red{\cC_{\rmSamp,2}}(1+\red{\cC_{\rmStab}}\rho_{X,\Omega}^{m-d/2})h_{X,\Omega}^{d/2-s}\|u\|_{\ell_2(X)}.
			\end{aligned}
		\end{equation*}
		Similarly, substituting \eqref{eq:DiscInvIneq_Hm_ell2} into \eqref{eq:SamplingIneq2} with $k=0$, we have
		$$\|u\|_{L_{\infty}(\Omega)}
		\le \red{\cC_{\rmSamp,\infty}} 
		\Big(\red{\cC_{\rmStab}} h_{X,\Omega}^{m-d/2}q_{X,\Omega}^{d/2-m}\|u\|_{\ell_2(X)}+\|u\|_{\ell_2(X)}\Big).$$
		This completes the proof.
	\end{proof}

	The crucial connection between the discrete $\ell_2(X)$ norms and the continuous $L_2(\Omega)$ norms relies on approximating integrals over bounded Lipschitz domains via quadrature rules with sufficiently dense node distributions. For bounded domains in $\R^{d}$, such results have been explored in \cite{wenzel2026sharp} using a geometrically greedy algorithm defined relative to a reference set. For completeness, we restate the key result below.
	\begin{lemma}[\textbf{$\bm{\ell_2}$-stability, \cite[Satz 2.1.6--7]{muller2009komplexitat} \& \cite[Theorem~7]{wenzel2026sharp}}]\label{lem:disctIneq}
		Suppose Assumption~\ref{Assump1} holds
		with an interior cone angle $\omega\in (0,\pi/2)$ and radius $r>0$.
		Given any reference set $Y_0$,
		we have the following estimates
		\begin{equation}\label{eq:Y0}
			h_{Y_0,\Omega}\geq c_{\Omega}N_{Y_0}^{-1/d},~\mbox{ and }~q_{Y_0,\Omega}\leq C_{\Omega}N_{Y_0}^{-1/d},
		\end{equation}
		with constants 
		\begin{equation}\label{eq:cOmega}
			c_{\Omega}=\pi^{-1/2}\Big(\mathrm{vol}(\Omega)\Gamma\big(\frac{d}{2}+1\big)\Big)^{1/d},
			~~C_{\Omega}=\big(\frac{2\pi}{\omega}\big)^{1/d}\pi^{-1/2}\Big(\mathrm{vol}(\Omega)\Gamma\big(\frac{d}{2}+1\big)\Big)^{1/d}.
		\end{equation}
		For any bounded function \red{$g\in C^0(\Omega)$}, and any constant $q_1\leq \min\big\{\frac{2^{-1/d}c_{\Omega}}{7C_{\Omega}}q_{Y_0},\frac{2}{5}r\big\}$, there exists a finite set of points $Y_{1,g}\subset\Omega$ with $\frac{1}{3}q_1\leq q_{Y_{1,g}}\leq h_{Y_{1,g}}\leq \frac{22}{3}q_1$, such that the following inequality holds
		$$\|g\|_{\ell_2(Y_{1,g})}\leq \red{\cC_{\rmStab,\ell_2}}N_{Y_{1,g}}^{1/2}\|g\|_{L_2(\Omega)},$$
		where $\red{\cC_{\rmStab,\ell_2}}=\sqrt{4\frac{16^dC_dC_{\Omega}^d}{c_{\Omega}^{2d}C_{d,\omega}^2}}$, with $C_d$ being the volume of the unit ball and $C_{d,\omega}$ the volume of the unit cone $C(x, \xi(x), \omega, 1)$.
	\end{lemma}

	The primary challenge in deriving the general Bernstein-type and Nikolskii-type inequalities lies in a significant technical limitation: we cannot directly apply the stability result from Lemma~\ref{lem:disctIneq} to establish the bound
	$\|u\|_{\ell_2(X)}\leq Ch_{X,\Omega}^{-d/2}\|u\|_{L_2(\Omega)}$,
	as it is simply untrue that the set of centers $X$ constitutes a valid  quadrature point set for a lower estimate. To address this issue, our approach combines the sampling inequality with the stability result in Lemma~\ref{lem:disctIneq}, enabling us to circumvent the need for a direct quadrature assumption on $X$.
	Then, by applying the sampling inequality in \eqref{eq:SamplingIneq2}, we can obtain the following Nikolskii inequality.
	\begin{theorem}[\textbf{Nikolskii inequality}]\label{thm:Nikolskii_domain}
		Suppose Assumption~\ref{Assump1} holds. \red{For any integer $k$ such that $0\leq k\leq \lceil m-d/2\rceil-1$}, there exists a constant  \red{$C=C_{k,\omega,r,d,\phi_m,\rmLip_\Omega}$} such that
		\red{\begin{equation}\label{eq:InvLinf_domain}
				|u|_{W^k_{\infty}(\Omega)}\leq C q_{X,\Omega}^{-k-d/2}\|u\|_{L_2(\Omega)},~~\forall u\in V_{X,\phi_m,\Omega}.
			\end{equation}
		}
	\end{theorem}
	\begin{proof}
		Using the $\ell_2$ stability result from Lemma~\ref{lem:disctIneq}, we establish that for any
		$u\in V_{X,\phi_m,\Omega}$, 
		it is possible to use $X$ as the reference set and construct a $u$-dependent quasi-uniform quadrature point set $Y_u\subset\Omega$ such that $q_{Y_u}\sim q_{X}$ and
		\begin{equation}\label{eq:l2toL2}
			\|u\|_{\ell_2(Y_u)}\le \red{\cC_{\rmStab,\ell_2}} N_{Y_u}^{1/2}\|u\|_{L_2(\Omega)}
			\le \red{\cC_{\rmStab,\ell_2}C_{\Omega}^{d/2}} q_{Y_u,\Omega}^{-d/2}\|u\|_{L_2(\Omega)},
		\end{equation}
		where the final inequality follows from \eqref{eq:Y0} and the quasi-uniformity of $Y_u$.
		Next, \red{we apply} the sampling inequality \eqref{eq:SamplingIneq2} to the set $Y_u$. For \red{$0\leq k\leq \lceil m-d/2\rceil-1$},
		we have
		\red{
			$$|u|_{W^k_{\infty}(\Omega)}
			\le \red{\cC_{\rmSamp,\infty}}\big(
			h_{Y_u,\Omega}^{m-k-d/2}\|u\|_{H^{m}(\Omega)}+h_{Y_u,\Omega}^{-k}\|u\|_{\ell_2(Y_u)}\big).$$}
		Now, we invoke the inverse inequality from Theorem~\ref{thm:InvEstimate} for all $u\in V_{X,\phi_m,\Omega}$ and apply the stability result \eqref{eq:l2toL2} to obtain
		\red{
			\begin{align*}
				|u|_{W^k_{\infty}(\Omega)}\le &~ \cC_{\rmSamp,\infty}\big(h_{Y_u,\Omega}^{m-k-d/2}\,\cC_{\rmFB}\, q_{X,\Omega}^{-m}\|u\|_{L_2(\Omega)}+h_{Y_u,\Omega}^{-k}\|u\|_{\ell_2(Y_u)}\big)\\
				\le &~ \cC_{\rmSamp,\infty}\Big(\cC_{\rmFB} \rho_{Y_u,\Omega}^{m-k-d/2}q_{Y_u,\Omega}^{-k-d/2}\|u\|_{L_2(\Omega)}+\cC_{\rmStab,\ell_2}C_{\Omega}^{d/2}h_{Y_u,\Omega}^{-k} q_{Y_u,\Omega}^{-d/2}\|u\|_{L_2(\Omega)}\Big)\\
				\le &~ \cC_{\rmSamp,\infty}\Big(\cC_{\rmFB}\rho_{Y_u,\Omega}^{m-k-d/2}+\cC_{\rmStab,\ell_2}C_{\Omega}^{d/2}\rho_{Y_u,\Omega}^{-k}\Big)q_{Y_u,\Omega}^{-k-d/2}\|u\|_{L_2(\Omega)}.
			\end{align*}
		}
		Finally, using the relation
		$q_{Y_u}\sim q_X$ completes the proof.
	\end{proof}
	
	By applying the sampling inequality \eqref{eq:SamplingIneq1} and repeating the same arguments in the proof of Theorem~\ref{thm:Nikolskii_domain}, we can obtain the following result.
	\begin{theorem}\label{thm:BernsteinIneq}
		Suppose Assumption~\ref{Assump1} holds.
		Then for any real $s$ with $0\leq s\leq \lfloor m\rfloor$ and real $t\in[0,m]$,
		there exists a constant \red{$C=C_{m,s,t,\omega,r,d,\phi_m,\rmLip_\Omega}$} such that
		\red{$$|u|_{H^s(\Omega)}\leq C\big(
			\,q_{X,\Omega}^{t-s}\|u\|_{H^{t}(\Omega)}+ q_{X,\Omega}^{-s}\|u\|_{L_2(\Omega)}\big)$$}
		holds for all trial functions $u\in V_{X,\phi_m,\Omega}$.
	\end{theorem}

	\smallskip

	\section{Inverse inequalities for restricted kernels on manifolds}
	\label{sec:Nikolskii_Manifold}
	
	We first list the standing hypotheses used throughout this section. Under these assumptions, the native space of the restricted kernel $\psi_\tau$ is equivalent to the Sobolev space $H^{\tau}(\M)$.
	%
	\begin{assumption}\label{Assump2}
		Assume that $\M\subset\R^{d}$ is a closed \red{smooth embedded} Riemannian manifold of codimension  $\red{\dco}$, i.e.\ $\dim\M=d_\M=d-\red{\dco}$.
		Assume further that $X\subset\M$ is a finite set with
		separation distance $q_{X,\M}$.
		Let $\phi_m$ be a positive definite kernel on $\R^{d}$ satisfying the Fourier–decay condition~\eqref{eq:FourierDecay}, and we define the restricted kernel
		$ \psi_\tau := \phi_m\big|_{\M\times\M}$  with $\tau := m-\frac{\red{\dco}}{2} > \frac{d_\M}{2}$.
		Finally, we define $V_{X,\psi_\tau,\M}$ to be the finite-dimensional trial space spanned by the  translates of the restricted kernel $\psi_\tau$ centered at points in $X$, as in~\eqref{eq:approx_space_restricted}.
	\end{assumption}

	Our approach relies on several norm-equivalence results between the manifold $\M$ and its ambient space, which allows us to transform inequalities in domains to manifolds. First, we require the following definition for the tubular neighborhood domain $\Omgdel$ that was studied in \cite{chen2020extrinsic,cheung-2018SISC-kernel,fuselier2012scattered}.

	
	
	
	\red{
		\begin{definition}[Tubular neighborhood]\label{Def:Tubular_domain}
			Let $\M\subset\R^{d}$ be a closed smooth embedded Riemannian manifold of dimension $d_\M$, equipped with the induced Riemannian metric.
			There exists $\delta_{\M}>0$ such that, for every $0<\delta<\delta_{\M}$, the set
			\begin{equation}\label{eq:tubular_domain_dist}
				\Omega_{\delta}:=\{\,x\in\R^{d}:\dist(x,\M)<\delta\,\}
			\end{equation}
			is a tubular neighborhood of $\M$. Let $N_y\M$ denote the normal space to $\M$ at $y$. The closest point projection map
			$\Rcp:\Omega_{\delta}\to\M$ is well-defined and smooth, and each $x\in\Omega_{\delta}$
			admits a unique decomposition
			\begin{equation}\label{eq:unique_normal_decomposition}
				x=\Rcp(x)+\nu(x),\quad \nu(x)\in N_{\Rcp(x)}\M,\quad \|\nu(x)\|<\delta.
			\end{equation}
			Equivalently, letting $N\M:=\{(y,\nu):y\in\M,\ \nu\in N_y\M\}$ be the normal bundle,
			the map
			\begin{equation}\label{eq:tubemap_normal_bundle}
				\cT:\{(y,\nu)\in N\M:\ \|\nu\|<\delta\}\to \Omega_\delta,\quad \cT(y,\nu)=y+\nu,
			\end{equation}
			is a diffeomorphism for all $0<\delta<\delta_{\M}$.
		\end{definition}
	}
	
	Similar to Lemma~\ref{lem:H2invineq}, we can use the extension operator, trace operator and band-limited interpolant to establish the Bernstein inverse inequality on manifolds from $H^{\tau}(\M)$ to $H^{\beta}(\M)$ with any real $\beta\in({d_\M}/{2},\tau]$.
	\begin{lemma}[\textbf{A Bernstein inequality}]\label{lem:BernsteinLemma_manifolds}
		Suppose Assumption~\ref{Assump2} holds.  For any real $\beta\in({d_\M}/{2},\tau]$, there exists a constant $C=C_{\beta,d,\psi_\tau,\M}>0$ such that
		$$\|u\|_{H^{\tau}(\M)}\leq Cq_{X,\M}^{-\tau+\beta}\|u\|_{H^{\beta}(\M)}$$
		holds for all trial functions $u\in V_{X,\psi_{\tau},\M}$.
	\end{lemma}
	\begin{proof}
		For $d_\M/{2}<\beta\leq \tau$, we have $m\geq \beta+\red{\dco}/2>d/2$. \red{By \cite[Section 25]{besov1979integral}, \cite[Section 7]{adams_2003_sobolev} and } \cite[Theorem 17]{fuselier2012scattered}, there exists a continuous extension operator $E_{\M}: H^{\tau}(\M) \to H^m(\R^d)$ and a continuous trace operator $T_{\M}: H^m(\R^d) \to H^{\tau}(\M)$ such that
		$f = T_{\M}E_{\M}f$ for any $f \in H^{\tau}(\M)$.
		By Lemma~\ref{lem:Bandlimitedfunc}, there exists a band-limited $X$-interpolatory surrogate $f_\sigma:=f_{\sigma,u_\tau,\beta+\frac{\red{\dco}}{2},\M}$ selected by using the extension $E_\M u_\tau  \in H^m(\R^d) $ with $u_\tau \in V_{X,\psi_\tau,\M}$ and $\sigma=q_{X,\M}^{-1}$ satisfying
		\begin{equation}\label{eq:band_limited_manifold1}
			\begin{aligned}
				\|f_\sigma\|_{H^{m}(\R^d)}\lesssim&~ q_{X,\M}^{-m+\beta+\frac{\red{\dco}}{2}}\|f_\sigma\|_{H^{\beta+\frac{\red{\dco}}{2}}(\R^d)}
				\lesssim q_{X,\M}^{-\tau+\beta }\|E_{\M}u_\tau\|_{H^{\beta+\frac{\red{\dco}}{2}}(\R^d)}.
			\end{aligned}
		\end{equation}
		Since $T_{\M}$ is continuous and $u_\tau \in V_{X,\psi_\tau,\M}$ interpolates $f_\sigma$ on $X\subset\M$, we have
		\[
		\|u_\tau\|_{H^\tau(\M)}
		\leq
		\| u_\tau-T_\M f_{\sigma}\|_{H^\tau(\M)} + \|T_{\M}f_{\sigma}\|_{H^\tau(\M)}
		\lesssim \|T_{\M}f_{\sigma}\|_{H^\tau(\M)}
		\lesssim \|f_{\sigma}\|_{H^m(\R^d)},
		\]
		by the orthogonality of the interpolant.
		Combining this result with \eqref{eq:band_limited_manifold1} and the continuity of $E_{\M}$, we show that
		\begin{align*}
			\|u_\tau\|_{H^\tau(\M)}\lesssim q_{X,\M}^{-\tau+\beta}\|E_{\M}u_\tau\|_{H^{\beta+\frac{\red{\dco}}{2}}(\R^d)}
			\lesssim 
			q_{X,\M}^{-\tau+\beta}\|u_\tau\|_{H^\beta(\M)},
		\end{align*}
		which completes the proof.
	\end{proof}

	Previous work by \cite[Lemma 3.1]{cheung-2018SISC-kernel} (for codimension $\dco=1$) and \cite[Lemma 2.1]{chen2020extrinsic} (for  arbitrary codimension) established a norm equivalence for functions constant along the normal direction, $u\circ\Rcp$, between the manifold $\M$ and its ambient space $\Omgdel$.
	Specifically, for any $f\in H^{\tau}(\M)$ and $\varsigma\in[0,\tau]$,
	the following norm equivalence holds
	\begin{equation}\label{eq:NormEquiv_uRcp}
		\|f\circ\Rcp\|_{H^{\varsigma}(\Omgdel)}\sim \delta^{\red{\dco}/2}\|f\|_{H^{\varsigma}(\M)},
	\end{equation}
	for all constant-along-normal extensions with some constants depending only on $d,\tau,\varsigma$ and $\M$.
	This enables the application of the Gagliardo–Nirenberg interpolation inequality to establish the Bernstein inequality on manifolds, mapping from $H^{\beta}(\M)$ to $H^\eta(\M)$, where \red{$\beta\in[0,\tau]$} and $\eta\in[0,\beta]$.

	\begin{theorem}[\textbf{Full Bernstein inequality}]\label{thm:BII_manifold_Hvarrho}
		Suppose Assumption~\ref{Assump2} holds. Then for any real $\red{\beta\in[0,\tau]}$ and real $\eta\in[0,\beta]$, there exists a constant $C=C_{\beta,d,\psi_\tau,\M}>0$ such that
		$$\|u\|_{H^{\beta}(\M)}\leq Cq_{X,\M}^{-\beta+\eta}\|u\|_{H^{\eta}(\M)}
		$$
		holds for all trial functions $u\in V_{X,\psi_{\tau},\M}$.
	\end{theorem}

	\begin{proof}
		\red{For $\beta\in(d_\M/2,\tau]$, we can use the Gagliardo-Nirenberg interpolation inequality in Lemma~\ref{lem:interp_lipschitz} for $u\circ\Rcp$ on $\Omgdel$ and the norm equivalence in \eqref{eq:NormEquiv_uRcp} to prove $\|u\|_{H^{\beta}(\M)}\le C q_{X,\M}^{-\beta+\eta}\|u\|_{H^{\eta}(\M)}$.  Then, following the arguments in Theorem~\ref{thm:InvEstimate}, we apply the interpolation inequality again
			to extend the inverse estimate to the full range $\beta\in[0,\tau]$.}
		
	\end{proof}
	
	
	
	\red{To derive the Nikolskii inequality
		using the same proof techniques as Theorem~\ref{thm:Nikolskii_domain}}, we require both a sampling inequality and an $\ell_2$ stability on manifolds.
	Although Wenzel  \cite{wenzel2026sharp} remarked  that the stability result should carry over to the manifold setting, no proof is provided.
	In addition,  we cannot  apply any  of the Bernstein inequalities from  the previous section to $u\circ\Rcp \not\in V_{X,\psi_\tau,\M}$.
	We must therefore establish another norm equivalence result just for
	trial functions in $V_{X,\psi_\tau,\M}$.
	\red{Before doing so, we first have to verify that the $H^m$ to $L_2$ Bernstein inequality holds uniformly with respect to $\delta$ for all tubular neighborhood $\Omega_\delta$, which have a special geometric structure that allows us to apply a covering argument using centers from the manifold. }

	\red{
		\begin{theorem}[\textbf{Bernstein inequality in the tubular neighborhood $\Omgdel$}]
			\label{thm:uniform-bernstein-tube}
			Suppose Assumption~\ref{Assump2} holds. Then for some $\delta = \mathcal{O}(q_{X,\mathbb{M}})$ and any real $\beta\in[0,m]$, there exists a constant $C=C_{m,\beta,d,\phi_m,\M}>0$,
			independent of $q_{X,\M}$ and $\delta$, such that
			$$
			\|u\|_{H^\beta(\Omega_\delta)}
			\le
			Cq_{X,\M}^{-\beta}\|u\|_{L_2(\Omega_\delta)}
			$$
			for all trial functions $u\in V_{X,\psi_{\tau},\M}$.
		\end{theorem}
	}

	\red{
		\begin{proof}
			Since $\M$ is compact and smooth, the Euclidean and
			geodesic distances are locally comparable. Hence there exists a constant
			$c_{\M}>0$ such that
			$\|\xi-\eta\|
			\ge c_{\M} \operatorname{dist}_{\M}(\xi,\eta)$
			for all sufficiently close $\xi,\eta\in X\subset\M$. Therefore,
			$$\|\xi-\eta\|
			\ge 2c_{\M}q_{X,\M},
			\quad \xi,\eta\in X,\ \xi\neq\eta .$$
			Choose $\gamma=\min\{c_\M, (2C_{d,m}\cC_0^2)^{-\frac{1}{2m-d}}\}$, where $C_{d,m}$ and $\cC_0$ are given in \eqref{eq:Hm-L2+Hm} below, and let the tubular radius be $\delta=\gamma q_{X,\M}$. Then,
			$B_{\delta}(\xi)\cap B_{\delta}(\eta)=\emptyset$ for $\xi\neq \eta$. Thus, the family $\{B_{\delta}(\xi):\xi\in X\}$
			is a pairwise disjoint family of interior balls in $\Omega_\delta$.
			
			Next, we establish a scaled Sobolev point estimate that is a variant of the local estimate in \cite[Lemma 4.15]{adams_2003_sobolev}. Since $m>d/2$, the Sobolev embedding theorem gives
			$|w(0)|^2 \le C_{d,m}\|w\|_{H^m(B_1(0))}^2$ for every continuous $w\in H^m(B_1(0))$.
			Moreover, by \cite[Theorem 5.2]{adams_2003_sobolev} (with $p=2$, $\epsilon_0=1$), after applying a change of variables $y\to a+ry$, we obtain the corresponding inequality
			\begin{equation}\label{eq:ScaledPointEstimate}
				|v(a)|^2\le C_{d,m}\left(r^{-d}\|v\|_{L_2(B_r(a))}^2+r^{2m-d}|v|_{H^m(B_r(a))}^2\right).
			\end{equation}
			
			We now apply this estimate to any trial function $u$ and any center $\xi\in X$ to obtain
			$$
			|u(\xi)|^2
			\le
			C_{d,m}
			\left(
			(\gamma q_{X,\Omega_\delta})^{-d}
			\|u\|_{L_2(B_{\gamma q_{X,\Omega_\delta}}(\xi))}^2
			+
			(\gamma q_{X,\Omega_\delta})^{2m-d}
			|u|_{H^m(B_{\gamma q_{X,\Omega_\delta}}(\xi))}^2
			\right).
			$$
			Summing over $\xi\in X$, and using the pairwise disjointness of these
			balls together with
			$B_{\gamma q_{X,\Omega_\delta}}(\xi)\subset \Omega_\delta$, gives
			\begin{equation}\label{eq:discreEstimate1}
				\|u\|_{\ell_2(X)}^2
				\le
				C_{d,m}\gamma^{-d}q_{X,\Omega_\delta}^{-d}
				\|u\|_{L_2(\Omega_\delta)}^2
				+
				C_{d,m}\gamma^{2m-d}q_{X,\Omega_\delta}^{2m-d}
				|u|_{H^m(\Omega_\delta)}^2 .
			\end{equation}

			By Lemma~\ref{lem:InvIneq}, there exists a constant $\cC_{\rmStab}=C_{d,\phi_m}>0$,
			independent of $q_{X,\Omega_\delta}$ and $\delta$, such that we can further bound \eqref{eq:discreEstimate1}  as follows:
			\begin{equation}\label{eq:Hm-L2+Hm}
				\begin{aligned}
					|u|_{H^m(\Omega_\delta)}^2\le &~\cC_{\rmStab}^2 q_{X,\Omega_\delta}^{d-2m}
					\|u\|_{\ell_2(X)}^2\\
					\le &~ C_{d,m}\cC_{\rmStab}^2\gamma^{-d}q_{X,\Omega_\delta}^{-2m}
					\|u\|_{L_2(\Omega_\delta)}^2+C_{d,m}\cC_{\rmStab}^2\gamma^{2m-d}
					|u|_{H^m(\Omega_\delta)}^2.
				\end{aligned}
			\end{equation}
			Since $\gamma$ is chosen such that $C_{d,m} \cC_{\rmStab}^2 \gamma^{2m-d}<\frac12$, this proves the $H^m(\Omega_\delta)$ to $L_2(\Omega_\delta)$ inverse estimate.
			
			Finally, since $\delta=\mathcal{O}(q_{X,\M})$, we have $q_{X,\Omega_\delta}\sim q_{X,\mathbb M}$. Then, for $0\le \beta\le m$, the proof is completed by applying the Gagliardo--Nirenberg inequality as in Theorem~\ref{thm:InvEstimate} with $t=0$.
		\end{proof}
	}
	
	\red{
		\begin{remark}
			We remark that the proof of the above theorem is specific to tubular
			neighborhoods and does not directly extend to arbitrary domains, since the set $X\subset \M$ remains separated from the boundary of $\Omgdel$. A related idea appears in
			\cite[Section 6]{rieger-2008phd-sampling}, where scaled domains are introduced to
			mitigate boundary effects.
		\end{remark}
	}

	\begin{lemma}\label{lem:EquivNorm_Manifold_Omgdel}
		Suppose Assumption~\ref{Assump2} holds. Then for some $\delta = \mathcal{O}(q_{X,\mathbb{M}})$ and any real $\beta$ with \red{$0\leq \beta\leq\min\{\tau, \tau+d_c/2 -1\}$}, the following norm equivalence
		\begin{equation}\label{eq:EquivNorm}
			\|u\|_{H^{\beta}(\Omgdel)}\sim \delta^{\red{d_c}/2}\|u\|_{H^{\beta}(\M)},
		\end{equation}
		holds for all trial functions $u\in V_{X,\psi_{\tau},\M}$.
	\end{lemma}

	\begin{proof}
		By Definition~\ref{Def:Tubular_domain},
		the tubular neighborhood $\Omega_\delta$ admits a diffeomorphism via the map $\mathcal T$. First, we consider $\beta=0$. Applying the coarea formula and
		the change of variables $x=y+\nu$ yields
		\begin{equation}\label{eq:IntManifold_Ineq2}
			\int_{\Omega_\delta}|u(x)|^2\,\d x
			=
			\int_{\mathbb M}\int_{B_\delta^y}
			|u(y+\nu)|^2
			|\det(J_{\mathcal T}(y,\nu))|\,\d\nu\,\d\mu(y),
		\end{equation}
		where \red{$B_\delta^y:=\{\nu\in N_y\mathbb M:\|\nu\|<\delta\}$},
		and $J_{\mathcal T}(y,\nu)$ is the Jacobian matrix of the map
		$\mathcal T$. Since $\mathbb M$ is compact and $\mathcal T$ is a
		diffeomorphism for $0<\delta<\delta_{\mathbb M}$, there exist constants
		$0<c_{\mathcal T}<C_{\mathcal T}<\infty$, independent of $\delta$, such
		that $c_{\mathcal T}\le |\det(J_{\mathcal T}(y,\nu))|\le C_{\mathcal T}$
		for all $y\in\mathbb M$ and $\nu\in B_\delta^y$. We write
		$u(y+\nu)=u(y)+\widehat u(y,\nu)$
		with
		\begin{equation}\label{eq:IntManifold_Ineq3}
			\widehat u(y,\nu)
			:=
			u(y+\nu)-u(y)
			=
			\int_0^1 \nabla u(y+t\nu)\cdot \nu\,\d t.
		\end{equation}
		Putting this into \eqref{eq:IntManifold_Ineq2}, we can use the upper Jacobian bound and $|B_\delta^{d_c}|=\omega_{d_c}\delta^{d_c}$
		to derive
		$$
		\begin{aligned}
			\|u\|_{L_2(\Omega_\delta)}^2
			&\le
			C_{\mathcal T}
			\int_{\mathbb M}\int_{B_\delta^y}
			|u(y+\nu)|^2\,\d\nu\,\d\mu(y)                                      \\
			&\le
			2C_{\mathcal T}
			\int_{\mathbb M}\int_{B_\delta^y}
			|u(y)|^2\,\d\nu\,\d\mu(y)
			+
			2C_{\mathcal T}
			\int_{\mathbb M}\int_{B_\delta^y}
			|\widehat u(y,\nu)|^2\,\d\nu\,\d\mu(y)                              \\
			&=
			2C_{\mathcal T}\omega_{d_c}\delta^{d_c}
			\|u\|_{L_2(\mathbb M)}^2
			+
			2C_{\mathcal T}
			\int_{\mathbb M}\int_{B_\delta^y}
			|\widehat u(y,\nu)|^2\,\d\nu\,\d\mu(y).
		\end{aligned}
		$$
		Applying the Cauchy-Schwarz inequality, we obtain
		\red{$$
			|\widehat u(y,\nu)|^2
			=
			\left|
			\int_0^1 \nabla u(y+t\nu)\cdot \nu\,\d t
			\right|^2
			\le
			\|\nu\|^2
			\int_0^1 |\nabla u(y+t\nu)|^2\,\d t,
			$$}
		and integrating over the normal ball gives
		$$
		\begin{aligned}
			\int_{\mathbb M}\int_{B_\delta^y}
			|\widehat u(y,\nu)|^2\,\d\nu\,\d\mu(y)
			&\le
			\delta^2
			\int_{\mathbb M}\int_{B_\delta^y}
			\int_0^1 |\nabla u(y+t\nu)|^2\,\d t\,\d\nu\,\d\mu(y)\\
			&\le c_{\mathcal{T}}^{-1}\delta^2\|\nabla u\|_{L_2(\Omega_\delta)}^2.
		\end{aligned}
		$$
		Combining the above results leads to
		\begin{equation}\label{eq:IntManifold_Ineq4}
			\|u\|_{L_2(\Omega_\delta)}
			\le
			C\delta^{d_c/2}\|u\|_{L_2(\mathbb M)}
			+
			C\delta\|\nabla u\|_{L_2(\Omega_\delta)}.
		\end{equation}

		Furthermore, since $u\in V_{X,\psi,\tau,\mathbb M}$, the function $u$ is also well-defined on the
		tubular neighborhood for sufficiently small $\delta$. Thus, by the
		Bernstein inequality from Theorem~\ref{thm:uniform-bernstein-tube}, we have
		$$
		\|\nabla u\|_{L_2(\Omega_\delta)}
		\le
		C q_{X,\Omega_\delta}^{-1}\|u\|_{L_2(\Omega_\delta)}.
		$$
		Substituting this into \eqref{eq:IntManifold_Ineq4}, we obtain
		$$
		\|u\|_{L_2(\Omega_\delta)}
		\le
		C\delta^{d_c/2}\|u\|_{L_2(\mathbb M)}
		+
		C\delta q_{X,\Omega_\delta}^{-1}\|u\|_{L_2(\Omega_\delta)}.
		$$
		Choose $\delta=O(q_{X,\mathbb M})$ sufficiently small so that
		$C\delta q_{X,\Omega_\delta}^{-1}\le \frac12$.
		Then, the second term can be absorbed into the left-hand side. Hence,
		$$
		\|u\|_{L_2(\Omega_\delta)}
		\le
		C\delta^{d_c/2}\|u\|_{L_2(\mathbb M)}.
		$$
		
		\red{For the lower bound, fix $y \in \mathbb{M}$ and regard $\nu \mapsto u(y+\nu)$
			as a function on the normal ball $B_\delta^y$. Applying the Sobolev point
			estimate~\eqref{eq:ScaledPointEstimate} on this $d_c$-dimensional ball gives
			\begin{equation}\label{eq:IntManifold_Ineq5}
				|u(y)|^2=|F_y(0)|^2\le
				C\left(
				\delta^{-d_c}\|F_y\|_{L_2(B_\delta^y)}^2
				+\delta^{2-d_c}|F_y|_{H^1(B_\delta^y)}^2
				\right).
		\end{equation}}
		Integrating \eqref{eq:IntManifold_Ineq5} over $y\in\mathbb M$, and using the uniform
		Jacobian bounds yields
		$$
		\|u\|_{L_2(\mathbb M)}^2
		\le
		C\left(
		\delta^{-d_c}\|u\|_{L_2(\Omega_\delta)}^2
		+
		\delta^{2-d_c}\|\nabla u\|_{L_2(\Omega_\delta)}^2
		\right).
		$$
		Again applying the Bernstein inequality in Theorem~\ref{thm:uniform-bernstein-tube}, we obtain
		$$
		\|u\|_{L_2(\mathbb M)}^2
		\le
		C\left(
		\delta^{-d_c}
		+
		\delta^{2-d_c}q_{X,\Omega_\delta}^{-2}
		\right)
		\|u\|_{L_2(\Omega_\delta)}^2.
		$$
		Because $\delta=O(q_{X,\mathbb M})$ and
		$q_{X,\Omega_\delta}\sim q_{X,\mathbb M}$, we conclude that
		$\|u\|_{L_2(\mathbb M)}^2
		\le
		C\delta^{-d_c}
		\|u\|_{L_2(\Omega_\delta)}^2$,
		Combining the above results, we complete the proof of the equivalence relation for the case $\beta=0$.
		
		The result for higher-order derivatives follows similarly by replacing $u$ with $D^{\nu}u$, where $|\nu|=\beta$. This yields the equivalence relation
		$|u|_{H^{\beta}(\Omgdel)} \sim \delta^{d_c/2}|u|_{H^{\beta}(\M)}.$
		In the proof of this equivalence with the same logical steps,  we require the Bernstein inequality $| u|_{H^{\beta+1}(\Omgdel)}\lesssim q_{X,\Omgdel}^{-(\beta+1)}\| u\|_{L_2(\Omgdel)}$
		to hold for $\beta+1$. \red{To apply Theorem~\ref{thm:uniform-bernstein-tube} for $\beta+1$, we need $\beta+1\leq m$. This implies that $0\leq \beta\leq m-1$, or equivalently $0\leq \beta\leq  \tau+d_c/2-1$ since $\tau=m-d_c/2$.} Note that fractional values of $\beta$ can be obtained using interpolation theory in Sobolev spaces (see, e.g., \cite{brenner2008mathematical}).
		Finally, applying the Sobolev norm equivalence from \cite[Proposition 2.2]{brezis2018gagliardo} $$\|u\|_{H^{\beta}(\Omgdel)}\sim\big(\|u\|_{L_2(\Omgdel)}^2+|u|_{H^{\beta}(\Omgdel)}^2\big)^{1/2},$$
		completes the proof.
	\end{proof}

	\begin{lemma}\label{lem:InvIneqManifold}
		Suppose Assumption~\ref{Assump2} holds.
		For any real $\beta$ with \red{$0\leq \beta\leq\min\{\tau, \tau+d_c/2 -1\}$} and real $\eta\in[0,\beta]$, there exists a constant $C=C_{\beta,\eta,d,\psi_\tau,\M}>0$ such that
		\begin{equation}\label{eq:InvIneqManifold1}
			|u|_{H^{\beta}(\M)}\leq C\Big(q_{X,\M}^{-\beta+\eta} \|u\|_{H^{\eta}(\M)}+ q_{X,\M}^{-\beta}\|u\|_{L_2(\M)}\Big)
		\end{equation}
		holds for all trial functions $u\in V_{X,\psi_{\tau},\M}$.
	\end{lemma}
	\begin{proof}
		We can apply the Bernstein inverse inequality from Lemma~\ref{thm:BernsteinIneq} to obtain the following estimate on $\Omgdel$,
		\begin{equation*}
			|u|_{H^{\beta}(\Omgdel)}\lesssim  q_{X,\Omgdel}^{-\beta+\eta}\|u\|_{H^{\eta}(\Omgdel)}+ q_{X,\Omgdel}^{-\beta}\|u\|_{L_2(\Omgdel)}.
		\end{equation*}

		With the equivalence from Lemma~\ref{lem:EquivNorm_Manifold_Omgdel}, the above inequality can be transformed into
		\begin{equation*}
			\begin{aligned}
				|u|_{H^{\beta}(\M)}\lesssim &~\delta^{-\red{d_c}/2}|u|_{H^{\beta}(\Omgdel)}\\
				\lesssim &~\delta^{-\red{d_c}/2} \Big(\delta^{\red{d_c}/2}q_{X,\Omgdel}^{-\beta+\eta}\|u\|_{H^{\eta}(\M)}+\delta^{\red{d_c}/2}q_{X,\Omgdel}^{-\beta}\|u\|_{L_2(\M)}\Big)\\
				\lesssim &~q_{X,\Omgdel}^{-\beta+\eta}\|u\|_{H^{\eta}(\M)}+q_{X,\Omgdel}^{-\beta}\|u\|_{L_2(\M)}.
			\end{aligned}
		\end{equation*}
		This completes the proof of \eqref{eq:InvIneqManifold1} by noting that $q_{X,\Omgdel}\sim q_{X,\M}$.
	\end{proof}
	

	\begin{theorem}[\textbf{Nikolskii inequality}]\label{thm:Nikol_Manifolds}
		Suppose Assumption~\ref{Assump2} holds.  Then there exists a constant $C=C_{d,\psi_\tau,\M}>0$ independent of $q_{X,\M}$ such that
		\begin{equation}\label{eq:InvLinf}
			\|u\|_{L_{\infty}(\M)}\leq Cq_{X,\M}^{-d_\M/2}\|u\|_{L_2(\M)}
		\end{equation}
		holds for all  trial functions $u\in V_{X,\psi_\tau,\M}$.
	\end{theorem}
	\begin{proof}
		We use the Nikolskii inequality for $\Omgdel$ from Theorem~\ref{thm:Nikolskii_domain} with $k=0$ to get
		\begin{align*}
			\|u\|_{L_{\infty}(\M)}\leq  \|u\|_{L_{\infty}(\Omgdel)}\lesssim q_{X,\Omgdel}^{-d/2}\|u\|_{L_2(\Omgdel)}.
		\end{align*}
		Then, the equivalence relation from Lemma~\ref{lem:EquivNorm_Manifold_Omgdel}
		for $\beta=0$ gives
		$\|u\|_{L_2(\Omgdel)}
		\sim\delta^{\red{d_c}/2}\|u\|_{L_2(\M)}.$
		Using $\delta\sim q_{X,\M}$ completes the proof.
	\end{proof}

	\smallskip
	\section{Conclusion}
	\label{sec:Conclusion}
	We extend the Bernstein and Nikolskii inequalities from the literature to kernel-based approximation spaces on bounded Lipschitz domains and compact Riemannian manifolds. Specifically, we establish two Bernstein inequalities and one Nikolskii inequality for Sobolev reproducing kernels in Euclidean domains, and we transfer these results to restricted kernels on embedded manifolds through a new norm equivalence for tubular neighborhoods.
	The new inverse inequalities enrich the analytical toolbox for kernel methods, enabling stability and error analyses that were previously unattainable. \red{These results provide a robust theoretical foundation for the construction of kernel-based approximation spaces suitable for meshless Galerkin methods \cite{kunemund2019high,sun-2025JSC-structure,wendland-Adv-2020solving}, which will be further developed in future research.}

	\subsection*{Acknowledgements}
	\red{We thank the two anonymous reviewers for their constructive and insightful suggestions that have enhanced the quality of this manuscript. The authors also thank T. Wenzel, whose ideas developed in \cite{Avesani_2025_sobolev} helped extend the almost‑full Bernstein 
		result in the submitted version to the full Bernstein inequality established in Theorem~\ref{thm:InvEstimate}.
	}
	
	\subsection*{Funding}
	The first author was supported in part by National Natural Science Foundation of China (No. 12571407), Basic Research Program of Jiangsu (No. BK20252037), and a Jiangsu Shuangchuang Team program (No. JSSCTD202449). The work of the second author was supported by the General Research Fund (GRF No. 12301824, 12300922) of Hong Kong Research Grant Council.

	\bibliographystyle{plain}
	\bibliography{Nikolskii}
	
\end{document}